\documentclass[draft,12pt,a4paper]{article}
\usepackage[hscale=0.7,vscale=0.77]{geometry}
\usepackage[dvips]{epsfig}
\usepackage{graphics}
\usepackage{amsmath}
\usepackage{amsfonts}
\usepackage{amssymb}
\usepackage{amsbsy}
\usepackage{amsthm}
\usepackage{mathabx}
\usepackage[mathscr]{eucal}
\usepackage[OT1,T1,TS1,T2A]{fontenc}
\usepackage[cp1251]{inputenc}
\usepackage[russian,english]{babel}

\newtheorem{theorem}{THEOREM}[section]
\newtheorem{lemma}{LEMMA}[section]
\newtheorem{definition}{DEFINITION}[section]
\newtheorem{remark}{REMARK}[section]
\newtheorem{corollary}{COROLLARY}[section]

\numberwithin{equation}{section}

\title{\textbf{On the completeness of Gaussians in a Hilbert functional space.}}
\author{\textbf{Victor Katsnelson}}
\date{}
\begin{document}

\maketitle

\vspace{-3.5ex}
\thanks{ \centerline{\footnotesize
Department of Mathematics, Weizmann Institute}}
\thanks{\footnotesize\centerline{ Rehovot, 7610001, Israel}}
\thanks{\footnotesize \centerline{
e-mail: \texttt{victor.katsnelson@weizmann.ac.il, victorkatsnelson@gmail.com}}}

\thispagestyle{empty}

\vspace{2.0ex}
\begin{abstract}
\noindent
Let \(w_{T}(t)\) and \(w_{\Omega}(\omega)\) be functions defined for \(t\in\mathbb{R}\)
and \(\omega\in\mathbb{R}\) respectively, where \(\mathbb{R}=(-\infty,\infty)\). The functions \(w_{T}(t)\) and
\(w_{\Omega}(\omega)\) are assumed to be non-negative everywhere and bounded away from
zero outside some sets of finite Lebesgue measure. Moreover some regularity conditions are posed
on the these functions. Under these regularity conditions, the functions \(w_{T}\) and
\(w_{\Omega}\) are locally bounded and
 grow not faster than exponentially.
We associate the inner product space \(\mathcal{H}_{w_T,w_{\Omega}}\) with the
functions \(w_{T}\) and \(w_{\Omega}\). The space \(\mathcal{H}_{w_T,w_{\Omega}}\) consists of those
functions \(x\)  for which
\(\int\limits_{-\infty}^{\infty}|x(t)|^2w_{T}(t)dt+
\int\limits_{-\infty}^{\infty}|\hat{x}(\omega)|^2w_{\Omega}(\omega)d\omega=
\|x\|^2_{\mathcal{H}_{w_T,w_{\Omega}}}<\infty\), where \(\hat{x}\)
is the Fourier transform of the function \(x\).
We show that the system of Gaussians
\(\big\lbrace\exp(-\alpha(t-\tau)^2)\big\rbrace\), where \(\alpha\) runs over \(\mathbb{R}^{+}=(0,+\infty)\) and \(\tau\) runs over \(\mathbb{R}\),
is a complete system in the space \(\mathcal{H}_{w_T,w_{\Omega}}\).

\end{abstract}

\vspace{2.0ex}
\noindent
{\small\textbf{Mathematics Subject Classification\ (2010).}\ Primary 46E20, 46E35; \ Secondary
41A28, 42B10.}

\vspace{1.0ex}

\normalsize
\vspace{3.0ex}
\noindent
\textsf{Notation:}\\
\(\mathbb{R}\) \ \ - the real axis, i.e. the set of all real numbers.\\
\(\mathbb{R}_{+}\) - the set of all \emph{strictly} positive real numbers.\\
\(\mathbb{N}\) \ \ - the set of all natural numbers.\\
\(d\xi,\,dt,\,d\omega\) - \begin{minipage}[t]{0.8\linewidth}
the normalized Lebesgue measure on the real axis \(\mathbb{R}\) (if the variable on \(\mathbb{R}\) is denoted by \(\xi,\,t,\,\omega\) respectively.)
\end{minipage}
\section{The space \(\boldsymbol{\mathcal{H}_{w_T,w_{\Omega}}}\).}
\subsection{The space \(\boldsymbol{L^2\pmb{(}\pmb{\mathbb{R}}, w\pmb{(}\xi\pmb{)}d\xi\pmb{)}}\).}
Let \(w(\xi)\) be a function defined  for all \(\xi\in\mathbb{R}\). In what follow we always assume that \(w(\xi)\) satisfies the conditions:
\begin{equation}
\label{NNe}
0\leq{}w(\xi)<\infty,\ \ \forall\,\xi\in\mathbb{R}.
\end{equation}
\centerline{\textit{A function \(w(\xi)\) satisfying the condition  \eqref{NNe} is said to be \textsf{the weight function}.}}
We interpret the variable \(\xi\) either as the time variable \(t\), or  as the frequency
variable~\(\omega\).
\begin{definition}
\label{defL2}
We associate the space \(L^2(\mathbb{R},w(\xi)d\xi)\) with the weight function \(w(\xi)\).
This is the space of all functions \(x(\xi)\) which are defined \(d\xi\)-almost everywhere on \(\mathbb{R}\), take values in the set \(\mathbb{C}\) and satisfy the condition
\begin{subequations}
\label{DefL2}
\begin{gather}
\label{DefL2a}
\|x(\xi)\|_{L^2(\mathbb{R},w(\xi)d\xi)} <\infty,\\[-3.0ex]
\intertext{where}
\notag\\[-6.5ex]
\label{DefL2b}
\|x(\xi)\|_{L^2(\mathbb{R},w(\xi)d\xi)}=\bigg(\int\limits_{\mathbb{R}}|x(\xi)|^2w(\xi)\,d\xi\bigg)^{1/2}.
\end{gather}
\end{subequations}
The functions \(x^{\prime}(\xi)\) and \(x^{\prime\prime}(\xi)\)
 such that \(\|x^{\prime}(\xi)-x^{\prime\prime}(\xi)\|_{L^2(\mathbb{R},w(\xi)d\xi)}=0\)
 determine the same element of the space \(L^2(\mathbb{R},w(\xi)d\xi)\).
\end{definition}
\vspace{2.0ex}
\noindent
 The space \(L^2(\mathbb{R},w(\xi)d\xi)\) is an inner product vector space.
 If \(x_1,x_2\in\!L^2(\mathbb{R},w(\xi)d\xi)\), then the inner product
 \(\langle{}x_1,x_2\rangle_{L^2(\mathbb{R},w(\xi)d\xi)}\) is defined as
\begin{equation}
\label{InnPr}
\langle{}x_1,x_2\rangle_{L^2(\mathbb{R},w(\xi)d\xi)}=
\int\limits
_{\mathbb{R}}x_1(\xi)\overline{x_2(\xi)}w(\xi)d\xi\,.
\end{equation}
The inner product \eqref{InPr} generates the norm \eqref{DefL2b}:
\[\|x\|^2_{L^2(\mathbb{R},w(\xi)d\xi)}=\langle{}x,x\rangle_{L^2(\mathbb{R},w(\xi)d\xi)},
 \ \ \forall\,x\in{}L^2(\mathbb{R},w(\xi)d\xi).\]

 \subsection{Fourier transform.}
For a complex-valued function \(x\) defined on \(\mathbb{R}\), its Fourier transform \(\hat{x}\) is
\begin{subequations}
\label{FoT}
\begin{equation}
\label{DiFoT}
\hat{x}(\omega)=\int\limits_{\mathbb{R}}x(t)e^{-2\pi{}i\omega{}t}dt\,,
\quad\omega\in\mathbb{R}.
\end{equation}
For a complex-valued function \(y\) defined on \(\mathbb{R}\), its inverse Fourier transform \(\check{y}\) is
\begin{equation}
\label{InFoT}
\check{y}(t)=\int\limits_{\mathbb{R}}y(\omega)e^{2\pi{}it\omega{}}d\omega\,,
\quad{t}\in\mathbb{R}.
\end{equation}
The transforms \eqref{DiFoT} and \eqref{InFoT} are mutually inverse:
\begin{equation}
\label{MuInv}
\textup{if } y=\hat{x}, \textup{ than } x=\check{y}.
\end{equation}
\end{subequations}
The Fourier transform and the inverse Fourier transform can be considered as operators \(\mathscr{F}\) and \(\mathscr{F}^{-1}\) acting in \(L^2(\mathbb{R})\):
\begin{subequations}
\label{FoOp}
\begin{align}
\label{DiFoOp}
\mathscr{F}x&=\hat{x}\,,\\
\mathscr{F}^{-1}y&=\check{y}\,.
\label{MuInV}
\end{align}
\end{subequations}
The Fourier operator \(\mathscr{F}\) is  well defined for \(x\in{}L^2(\mathbb{R},dt)\) and maps
the space \(L^2(\mathbb{R},dt)\) \emph{onto} the space \(L^2(\mathbb{R},d\omega)\) isometrically: the Parseval equality
\begin{equation}
\label{Par}
\|x\|^2_{L^2(\mathbb{R},dt)}=
\|\hat{x}\|^2_{L^2(\mathbb{R},d\omega)}, \ \ \forall\,x\in{}L^2(\mathbb{R},dt)
\end{equation}
holds.

The equalities \eqref{FoOp} can be expressed as
\begin{equation}
\label{MuIn}
\mathscr{F}^{-1}\mathscr{F}=\mathscr{I}_{T},\quad \mathscr{F}\,\mathscr{F}^{-1}=\mathscr{I}_{\Omega},
\end{equation}
where \(\mathscr{I}_{T}\) and \(\mathscr{I}_{\Omega}\) are the identity operators
in the spaces \(L^2(\mathbb{R},dt)\) and \(L^2(\mathbb{R},d\omega)\)
respectively.
\subsection{Definition of the space \(\boldsymbol{\mathcal{H}_{w_T,w_{\Omega}}}\).}
Let \(w_{T}(t)\) and \(w_{\Omega}(\omega)\) be two weight functions defined for \(t\in\mathbb{R}\)
and \(\omega\in\mathbb{R}\) respectively.
We associate the spaces \(L^2(\mathbb{R},w_T(t)dt)\) and \(L^2(\mathbb{R},w_{\Omega}(\omega)d\omega)\) with these weight functions.

 \vspace{2.0ex}
 \noindent
 \begin{definition}
 \label{DMS}
 \begin{subequations}
  Let weight functions \(w_{T}(t)\) and \(w_{\Omega}(\omega)\) be given.\\
   \textsf{The space} \(\boldsymbol{\mathcal{H}_{w_T,w_{\Omega}}}\) is defined as the set of all those
  \(x(t)\in{}L^2(\mathbb{R}, dt)\) for which
  \begin{gather}
  x(t)\in{}L^2(\mathbb{R},w_T(t)dt) \textup{ \ and \ } \hat{x}(\omega)\in{}L^2(\mathbb{R},w_{\Omega}(\omega)d\omega),
  \end{gather}
  where \(\hat{x}(\omega\) is the Fourier transform of the function \(x(t)\): \(\hat{x}=\mathscr{F}x\).

 \textup{Equivalently}\\
  \textsf{The space} \(\boldsymbol{\mathcal{H}_{w_T,w_{\Omega}}}\) is defined as the set of all those \(y(\omega)\in{}L^2(\mathbb{R}, d\omega)\) for which
  \begin{gather}
  y(\omega)\in{}L^2(\mathbb{R},w_{\Omega}(\omega)d\omega) \textup{ and } \check{y}(t)\in{}L^2(\mathbb{R},w_{T}(t)dt),
  \end{gather}
  \end{subequations}
  where \(\check{y}(t)\) is the inverse Fourier transform of the function \(y(\omega)\):
  \(\check{y}=\mathscr{F}^{-1}y\).
 \end{definition}

 The space \(\mathcal{H}_{w_T,w_{\Omega}}\) is an inner product space.
 If \(x_1,\,x_2\in\mathcal{H}_{w_T,w_{\Omega}}\), then the inner product
 \(\langle{}x_1,x_2\rangle_{\mathcal{H}_{w_T,w_{\Omega}}}\) is defined as
 \begin{equation}
 \label{InPr}
 \langle{}x_1,x_2\rangle_{\mathcal{H}_{w_T,w_{\Omega}}}
 =\langle{}x_1,x_2\rangle_{L^2(\mathbb{R},w_T(t)dt)}
 +\langle{}\widehat{x_1},\widehat{x_2}\rangle_{L^2(\mathbb{R},w_{\Omega}(\omega)dt)}\,.
 \end{equation}
In particular, the expression \(\|x\|^{2}_{\mathcal{H}_{w_T,w_{\Omega}}}\) for the square of norm of \(x\in\mathcal{H}_{w_T,w_{\Omega}}\) is
\begin{equation}
\label{IOW}
\|x\|^{2}_{\mathcal{H}_{w_T,w_{\Omega}}}=\|x\|^{2}_{L^2(\mathbb{R},w_T(t)dt)}+
\|\hat{x}\|^{2}_{L^2(\mathbb{R},w_{\Omega}(\omega)d\omega)}.
\end{equation}
 \begin{remark}
 If the weight functions \(\,w_{T}(t)\), \(w_{\Omega}(\omega)\) grow very fast as
\(|t|\to\infty,\,|\omega|\to\infty\), then the space \(\mathcal{H}_{w_{T},w_{\Omega}}\) may consist of  the identically zero function only.

For example, if \(w_{T}(t)=\exp(\alpha{}t^2),\,w_{\Omega}(\omega)=\exp(\beta{}\omega^2)\), where
\(\alpha>0,\,\beta>0\), and \(\alpha\beta>4\pi^2\), then the space
\(\mathcal{H}_{w_{T},w_{\Omega}}\) contains
only the identically zero function. This statement is a version of Hardy's Theorem\,\footnote{
Concerning this theorem, we  refer to \cite[Sec.\,3.2]{DyMcK}.}.
 \end{remark}
 \noindent
\textbf{EXAMPLE}.
Let
\vspace{-0.5ex}
\begin{equation}
\label{SWF}
w_{T}(t)\equiv1 \ \textup{ and } \ w_{\Omega}(\omega)=(2\pi\omega)^{2n}.
\end{equation}
For \(x(t)\in{}L^2(\mathbb{R},dt)\),
the condition \(\int\limits_{\mathbb{R}}|\hat{x}(\omega)|^2{\omega}^{2n}d\omega<\infty\)
means that the \(n\)-th derivative \(x^{(n)}(t)\) of the function \(x(t)\), considered as a \textsf{distribution} originally, is actually a \textsf{function} from \(L^2(\mathbb{R},dt)\). Moreover
\begin{equation*}
\int\limits_{\mathbb{R}}\big|\hat{x}(\omega)\big|^2w_{\Omega}(\omega)d\omega=
\int\limits_{\mathbb{R}}\big|x^{(n)}(t)\big|^2dt.
\end{equation*}
Thus for the weight functions of the form \eqref{SWF},
\begin{equation*}
\|x\|^2_{\mathcal{H}_{w_{T},w_{\Omega}}}=\int\limits_{\mathbb{R}}|x(t)|^2dt+
\int\limits_{\mathbb{R}}\big|x^{(n)}(t)\big|^2dt.
\end{equation*}
So for the weight functions of the form \eqref{SWF}, the space \(\mathcal{H}_{w_{T},w_{\Omega}}\)
is the Sobolev space \(W_{\;2}^{(n)}\).
 \subsection{The non-degeneracy condition.}
 \begin{definition} \
 \label{DeCo}
 The weight functions \(w_{T}(t)\)  \textsf{satisfy
 the non-de\-generacy condition} if there exist  number
 \(\varepsilon_T>0\)  such that the set
 \(E_{T}=\lbrace{}t:\,w_{T}(t)<\varepsilon_{T}\rbrace\)
 is of finite Lebesgue measure:
 \(\textup{mes}\,E_{T}<\infty\).

 The weight functions \(w_{\Omega}(\omega)\)  \textsf{satisfy
 the non-de\-generacy condition} if there exist  number
 \(\varepsilon_{\Omega}>0\)  such that the set
 \(E_{\Omega}=\lbrace{}t:\,w_{\Omega}(\omega)<\varepsilon_{\Omega}\rbrace\)
 is of finite Lebesgue measure:
 \(\textup{mes}\,E_{\Omega}<\infty\).
 \end{definition}
\begin{theorem} \
If both weight functions \(w_{T}\) and \(w_{\Omega}\) satisfy the non-degeneracy conditions,
then the the inequality
\begin{equation}
\label{MaNo}
\|x\|^2_{L^2(\mathbb{R},dt)}\leq{}B\|x\|^2_{\mathcal{H}_{w_{T},w_{\Omega}}}
\end{equation}
holds for every  \(x\in\mathcal{H}_{w_{T},w_{\Omega}}\), where \(B<\infty\) is
a constant which does not depend on \(x\).
\end{theorem}
\begin{proof}
In \cite{Naz} F.\,Nazarov proved the remarkable inequality
\begin{equation}
\label{NazIn}
\int\limits_{\mathbb{R}}|x(t)|^2dt\leq{}A\exp\{A\cdot\textup{mes}\,E\cdot\textup{mes}\,F\,\}
\Big(\int\limits_{\mathbb{R}\setminus\,E}|x(t)|^2dt+
\int\limits_{\mathbb{R}\setminus\,F}|\hat{x}(\omega)|^2d\omega\Big)\,,
\end{equation}
where \(x\) is an arbitrary function from \(L^2(\mathbb{R})\),
\(\hat{x}\) is the Fourier transform of the function \(x\),
\(E\) and \(F\) are
arbitrary subsets of \(\mathbb{R}\), and \(A\), \(0<A<\infty\), is an absolute constant.
(The value \(A\) does not depend on \(x,\,E,\,F\).) We apply the Nazarov inequality
\eqref{NazIn} to the sets \(E=E_T\),\,\(F=E_\Omega\)
which appear in Definition \ref{DeCo}
and to  an arbitrary function  \(x\in\mathcal{H}_{w_{T},w_{\Omega}}\).
Since \(\varepsilon_{T}\leq{}w_{T}(t)\) for \(t\in\mathbb{R}\setminus{}E_T\)
and \(\varepsilon_{\Omega}\leq{}w_{\Omega}(\omega)\) for
\(\omega\in\mathbb{R}\setminus{}E_{\Omega}\), the inequalities
\begin{equation*}
\int\limits_{\mathbb{R}\setminus{}E_{T}}|x(t)|^2dt\leq
\frac{1}{\varepsilon_T}\int\limits_{\mathbb{R}}|x(t)|^2w_T(t)dt,\
\int\limits_{\mathbb{R}\setminus{}E_{\Omega}}
|\hat{x}(\omega)|^2d\omega\leq
\frac{1}{\varepsilon_F}\int\limits_{\mathbb{R}}|\hat{x}(\omega)|^2w_{\Omega}(\omega)d\omega\,.
\end{equation*}
hold.
Combine these inequalities with \eqref{NazIn}, we come to the inequality
\eqref{MaNo}.
\end{proof}
\begin{corollary}
\label{Cor}
For every \(x\in\mathcal{H}_{w_{T},w_{\Omega}}\), the inequalities
\begin{subequations}
\label{CorET}
\begin{align}
\|x\|^2_{L^2(\mathbb{R},(1+w_{T}(t))dt)}\leq&(B+1)\|x\|^2_{\mathcal{H}_{w_{T},w_{\Omega}}},\\
\label{CorEO}
\|\hat{x}\|^2_{L^2(\mathbb{R},(1+w_{\Omega}(\omega))d\omega)}\leq&(B+1)\|x\|^2_{\mathcal{H}_{w_{T},w_{\Omega}}}
\end{align}
\end{subequations}
hold, where \(B\) is the same constant that in \eqref{MaNo}.
\end{corollary}
\begin{theorem}
\label{Com}
If both weight functions \(w_{T}(t)\) and \(w_{\Omega}(\omega)\) satisfy the non-degeneracy conditions, then the space \(\mathcal{H}_{w_{T},w_{\Omega}}\) provided by the norm \eqref{IOW} is complete.
\end{theorem}
\begin{proof} Let \(\lbrace{}x_n\rbrace_n\) be a Cauchy sequence of elements from
\(\mathcal{H}_{w_{T},w_{\Omega}}\), that is
\begin{equation*}
\|x_n-x_m\|_{\mathcal{H}_{w_{T},w_{\Omega}}}\to0 \ \textup{ as } m\to\infty,\,n\to\infty.
\end{equation*}
From the inequality \eqref{MaNo} and from the Parseval equality it follows that
\begin{equation*}
\|x_n-x_m\|_{L^2(\mathbb{R},dt)}\to0 \ \textup{ and } \|\hat{x}_n-\hat{x}_m\|_{L^2(\mathbb{R},d\omega)}\to0
\textup{ as } m\to\infty,\,n\to\infty.
\end{equation*}
Since the space \(L^2(\mathbb{R},dt)\) is complete, there exists \(x\in{}L^2(\mathbb{R},dt)\)
such that
\begin{equation*}
\|x_n-x\|_{L^2(\mathbb{R},dt)}\to0 \ \textup{ as } n\to\infty.
\end{equation*}
In view of the Parseval equality,
\begin{equation*}
\|\hat{x}_n-\hat{x}\|_{L^2(\mathbb{R},d\omega)}\to0 \ \textup{ as } n\to\infty.
\end{equation*}
According to well known results from the measure theory\footnote{See for example
\cite[Sections 21,22]{Hal}}, we can select an increasing subsequence \(\lbrace{}n_k\rbrace\)
of natural numbers such that
\begin{subequations}
\label{CAE}
\begin{align}
\label{CAET}
x_{n_k}(t)\to x(t) \ &\textup{ as \(k\to\infty \ dt\)-almost everywhere on } \mathbb{R},\\
\label{CAEO}
 \hat{x}_{n_k}(\omega)\to\hat{x}(\omega) \ &\textup{ as
 \(k\to\infty \ d\omega\)-almost everywhere on } \mathbb{R}.
\end{align}
\end{subequations}
Since a Cauchy sequence is bounded,
\begin{equation*}
\int\limits_{\mathbb{R}}|x_{n_{k}}(t)|^2w_{T}(t)dt+
\int\limits_{\mathbb{R}}|\hat{x}_{n_{k}}(\omega)|^2w_{\Omega}(\omega)d\omega\leq{}C<\infty,
\ \ k=1,2,\,\ldots\,,
\end{equation*}
where \(C\) does not depend on \(k\). Using \eqref{CAE} and Fatou's Lemma, we conclude
that
\begin{equation*}
\int\limits_{\mathbb{R}}|x(t)|^2w_{T}(t)dt+
\int\limits_{\mathbb{R}}|\hat{x}(\omega)|^2w_{\Omega}(\omega)d\omega\leq{}C.
\end{equation*}
Thus \(x\in\mathcal{H}_{w_{T},w_{\Omega}}\).

Given \(\varepsilon>0\), we choose \(K(\varepsilon)<\infty\) such that
\begin{equation*}
\Big(k>K(\varepsilon), l>K(\varepsilon)\Big)\Rightarrow
\big\|x_{n_k}-x_{n_l}\big\|^2_{\mathcal{H}_{w_{T},w_{\Omega}}}\leq\varepsilon^2,
\end{equation*}
i.e.
\begin{equation*}
\int\limits_{\mathbb{R}}\big|x_{n_k}(t)-x_{n_l}(t)\big|^2w_{T}(t)dt +
\int\limits_{\mathbb{R}}\big|\widehat{x_{n_k}}(\omega)-\widehat{x_{n_l}}(\omega)\big|^2
w_{\Omega}(\omega)d\omega\leq\varepsilon^2.
\end{equation*}
Using \eqref{CAE}, we pass to the limit as \(l\to\infty\) in the last inequality.
By Fatou's Lemma, we obtain
\begin{equation*}
\int\limits_{\mathbb{R}}\big|x_{n_k}(t)-x(t)\big|^2w_{T}(t)dt +
\int\limits_{\mathbb{R}}\big|\widehat{x_{n_k}}(\omega)-\widehat{x}(\omega)\big|^2
w_{\Omega}(\omega)d\omega\leq\varepsilon^2,
\end{equation*}
i.e.
\begin{equation*}
\Big(k\geq{}K(\varepsilon)\Big)\Rightarrow
\big\|x_{n_k}-x\big\|^2_{\mathcal{H}_{w_{T},w_{\Omega}}}\leq\varepsilon^2.
\end{equation*}
\end{proof}

\subsection{The regularity condition.}
In what follows we impose some regularity condition on the weight function \(w(\xi)\).
In particular, this regularity condition ensure  that the shift operator \(\mathscr{T}_h\)
is a bounded operator in the space \(L^2(\mathbb{R},w(\xi)d\xi)\) for each \(h\in\mathbb{R}\).

The regularity condition is formulated un terms of the value \(M_{w}(\delta)\) which can be
interpreted as the multiplicative modulus  of continuity (m.m.c.) of the weight function \(w(\xi)\).
\begin{definition}
\label{DMMC}
Let \(w(\xi)\) be a weight function satisfying the condition
\begin{equation}
\label{StriPos}
w(\xi)>0, \ \ \forall\,\xi\in\mathbb{R}.
\end{equation}
 \textsf{The  multiplicative modulus  of continuity} \(M_{w}(\delta)\) of the weight function \(w(\xi)\)
is defined as
\begin{gather}
\label{RegVa}
M_{w}(\delta)\stackrel{\textup{def}}{=}
\sup\limits_{\substack{\xi^{\prime}\in\mathbb{R},
\,\xi^{\prime\prime}\in\mathbb{R}\\|\xi^{\prime}-\xi^{\prime\prime}|\leq\delta}}
\frac{w(\xi^{\prime})}{w(\xi^{\prime\prime})}\,,\hspace*{3.0ex}0\leq\delta<\infty.
\end{gather}
The function \(M_{w}(\delta)\) may take the value \(\infty\):
\(M_{w}(\delta)\leq\infty\) for \(\delta\in[0,\infty)\).
\end{definition}
\begin{remark}
Without condition
\eqref{StriPos}, the value \eqref{RegVa} may be not well defined: considering the ratio
\(\frac{w(\xi^{\prime})}{w(\xi^{\prime\prime})}\), we may come to the uncertainty of the form~%
\(\frac{0}{0}\).
\end{remark}
\begin{remark}
\label{winv}
Under the condition \eqref{StriPos},
the function \(w^{-1}(\xi)=1/w(\xi)\) also is a weight function.
From \eqref{RegVa} it is clear that  multiplicative modules  of continuity of these weight functions
coincide:
\begin{equation}
\label{CoInc}
M_{w}(\delta)=M_{w^{-1}}(\delta), \ \ \forall\,\delta\in[0,\infty).
\end{equation}
\end{remark}
\begin{remark}
The function \(M_w(\delta)\) is not necessarily continuous with respect to
 \(\delta\).
\end{remark}
\begin{lemma}
\label{PMMC}
For any\footnote{To exclude division by zero, we must assume that the condition \eqref{StriPos} holds.} weight function \(w(\xi)\), its multiplicative modulus  of continuity \(M_{w}(\delta)\)
possesses the properties
\begin{enumerate}
\item[\textup{1}.] {\ }\\[-5.5ex]
\begin{equation}
\label{MMCz}
M_{w}(0)=1;
\end{equation}
\item[\textup{2}.] The function \(M_{w}(\delta)\) increases with respect to \(\delta\):
\begin{gather}
\label{Mon}
\text{If \(0\leq\delta^{\prime}<\delta^{\prime\prime}<\infty\), then
\(M_w(\delta^{\prime})\leq{M_w(\delta^{\prime\prime})}\).}\\[-2.0ex]
\intertext{In particular,}
\notag\\[-6.0ex]
 1\leq{}M_w(\delta)\,, \ \ \forall\,\delta\in\mathbb{R}_{+}.
\end{gather}
\item[\textup{3}.] The functions \(M_w\) is submultiplicative:
\begin{gather}
\label{SuMuPr}
M_w(\delta_1+\delta_2)\leq{}M_w(\delta_1)\cdot{}M_w(\delta_2), \ \ \forall\,\delta_1\in\mathbb{R}_{+},\,\forall\,\delta_2\in\mathbb{R}_{+};\\[-1.5ex]
\intertext{In particular,}
\notag\\[-6.0ex]
 \label{SubSp}
 M_w(n\delta)\leq{}(M_w(\delta))^n,\ \
\forall\,n\in\mathbb{N},\,\forall\,\delta\in\mathbb{R}_{+}.
\end{gather}
\end{enumerate}
\end{lemma}

From \eqref{Mon} and \eqref{SubSp} it follows that only two (mutually exclusive) possibilities.
can happen:

\centerline{Either \(M_{w}(\delta)<\infty \ \forall\,\delta\in\mathbb{R}_{+}\), or
\(M_{w}(\delta)=\infty \ \forall\,\delta\in\mathbb{R}_{+}\).}

\begin{definition}
\label{DefReCo}
\textsf{The weight functions \(w(\xi)\)  satisfies the regularity condition} if
\begin{subequations}
\label{FRC}
\begin{gather}
\label{FRC1}
M_{w}(\delta)<\infty, \ \ \forall\,\delta\in\mathbb{R}_{+},\\[-2.5ex]
\intertext{or, what is the same,}
\notag\\[-6.0ex]
\label{FRC2}
\exists\,\delta\in\mathbb{R}_{+}: \ M_{w}(\delta)<\infty.
\end{gather}
\end{subequations}
\end{definition}

\begin{remark}
If the regularity condition for a weight functions \(w(\xi)\) is satisfied, than
its multiplicative continuity modulus \(M_{w}(\delta)\)
 is locally bounded and grow not faster then exponentially with respect to \(\delta\).
 Moreover the estimate
 \begin{equation}
 \label{MCE}
 M_{w}(\delta)\leq{}C_{w}\exp(\mu_{w}\delta), \ \ \forall\,\delta\in\mathbb{R}_{+}
 \end{equation}
 holds, where \(C_w\) and \(\mu_w\) are some constants, \(0<C_{w}<\infty,\ 0\leq\mu_{w}<\infty\).

  The estimate \eqref{MCE}
 can be easily derived from \eqref{SubSp}.

 Choosing \(\xi^{\prime}=\xi,\,\xi^{\prime\prime}=0\) in \eqref{RegVa}, we obtain the estimate
 \begin{equation}
 \label{MwE}
 w(\xi)\leq{}w(0)M_{w}(|\xi|), \ \ \forall\,\xi\in\mathbb{R}.
 \end{equation}
 Comparing \eqref{MCE} and \eqref{MwE}, we obtain the estimate of the weight function \(w(\xi)\) from above:
 \begin{equation}
 \label{McEa}
 w(\xi)\leq{}w(0)\,C_{w}\exp(\mu_{w}|\xi|), \ \ \forall\,\xi\in\mathbb{R}.
 \end{equation}
 Taking into account \textup{Remark \ref{winv}}, we obtain the estimate of the weight function
 \(w(\xi)\) from below:
 \begin{equation}
 \label{McEb}
w(0)\,C^{-1}_{w}\exp(-\mu_{w}|\xi|)\leq{}w(\xi), \ \ \forall\,\xi\in\mathbb{R}.
 \end{equation}
 The constants \(C_{w}\) and \(\mu_{w}\) in \eqref{McEa} and \eqref{McEb} are the same that in
 \eqref{MCE}.
\end{remark}
\begin{remark}
\label{IWF}
In what follows, we impose the regularity conditions on the weight functions
\(1+w_T(t)\) and \(1+w_\Omega(\omega)\), where \(w_T\) and \(w_{\Omega}\) are
the same weight functions which appear in the definition\footnote{%
See Definition \ref{DMS}.
} %
of the space
\(\mathcal{H}_{w_T,w_{\Omega}}\).
These regularity conditions imply the estimates
\begin{subequations}
\label{IWFE}
\begin{align}
\label{IWFEa}
w_T(t)&\leq{}C_Te^{\mu_T|t|}, \ \ t\in\mathbb{R},\\
\label{IWFEb}
w_\Omega(\omega)&\leq{}C_{\Omega}e^{\mu_\Omega|\omega|},
 \ \ \omega\in\mathbb{R},
\end{align}
where \(C_T,\infty\), \(C_\Omega,\infty\), \(\mu_T<\infty\), \(\mu_{\Omega}<\infty\)
are some constants.
\end{subequations}
\end{remark}
\subsection{The shift operator.}
\begin{definition}
 \label{DefSh}
 For a function \(x(\xi)\) defined for \(\xi\in\mathbb{R}\) and for a number \(h\in\mathbb{R}\),
 the \textsf{shifted function} \((\mathscr{T}_hx)(\xi)\) is defined as a function
 \begin{equation}
 \label{DefShf}
 (\mathscr{T}_{\eta}x)(\xi)=x(\xi-\eta), \ \ \forall\,\xi\in\mathbb{R}.
 \end{equation}
 of the variable \(\xi\).
 The number \(\eta\) is considered as a parameter, the so called \textsf{shift parameter}.
 \end{definition}

 \begin{lemma}
 \label{PrTrO}
 Let \(w(\xi)\) be a weight function which satisfy the regularity condition~%
 \footnote{%
 \textup{See Definition \ref{DefReCo}.}
 }.%
 Then
 \begin{enumerate}
 \item[\textup{1}.]
 For each value \(\eta\in\mathbb{R}\) of the shift parameter, the shift operator \(\mathscr{T}_{\eta}\) is a bounded operator in the space \(L^2(\mathbb{R},w(\xi)d\xi)\).
 The inequality
 \begin{equation}
 \label{ENS}
 \|(\mathscr{T}_{\eta}x)(\xi)\|_{L^2(\mathbb{R},w(\xi)d\xi)}\leq
 \sqrt{M_{w}(|\eta|)}\,\|x(\xi)\|_{L^2(\mathbb{R},w(\xi)d\xi)}
 \end{equation}
 holds for every \(x(\xi)\in{}L^2(\mathbb{R},w(\xi)d\xi)\) and every \(\eta\in\mathbb{R}\), where \(M_{w}(\eta)\) is the multiplicative continuity modulus of the weight function \(w\).
 \item[\textup{2}.] The operator function \(\mathscr{T}_{\eta}:
 L^2(\mathbb{R},w(\xi)d\xi)\to{}L^2(\mathbb{R},w(\xi)d\xi)\) is strongly continuous
 with respect to \(\eta\):
 \begin{equation}
 \label{StCo}
 \lim_{|\eta|\to0}\|(\mathscr{T}_{\eta}x)(\xi)-x(\xi)\|_{L^2(\mathbb{R},w(\xi)d\xi)}=0,
 \ \ \forall\,x\in{}L^2(\mathbb{R},w(\xi)d\xi)\,.
 \end{equation}
 \end{enumerate}
\end{lemma}
\begin{proof} {\ }\\
 \textbf{1}. According to definition of the norm in \(L^2(\mathbb{R},w(\xi)d\xi)\),
\[\|(\mathscr{T}_{\eta}x)(\xi)\|^2_{L^2(\mathbb{R},w(\xi)d\xi)}=
\int\limits_{\mathbb{R}}|x(\xi-\eta)|^2w(\xi)d\xi.
\]
Changing variable \(\xi-\eta\to\xi\), we obtain
\[
\|(\mathscr{T}_{\eta}x)(\xi)\|^2_{L^2(\mathbb{R},w(\xi)d\xi)}=
\int\limits_{\mathbb{R}}|x(\xi)|^2w(\xi+\eta)d\xi.
\]
According to Definition \ref{DMMC}, the inequality
\[w(\xi+\eta)\leq{}w(\xi)M_{w}(|\eta|)\]
holds. Combining this inequality with the previous equality, we obtain \eqref{ENS}.\\
\textbf{2}. Given a function \(x(\xi)\in{}L^2(\mathbb{R},w(\xi)d\xi)\) and \(\varepsilon>0\),
we can split the function \(x\) into the sum
\[x(t)=\varphi(t)+r(t),\]
where \(\varphi\) is a continuous function with a compact support,
and
\[\|r\|_{L^2(\mathbb{R},w(\xi)d\xi)}<\varepsilon.\]
Hence
\begin{multline*}
\|(\mathscr{T}_{\eta}x)(\xi)-x(\xi)\big\|_{L^2(\mathbb{R},w(\xi)d\xi)}\\
\leq\|(\mathscr{T}_{\eta}\varphi)(\xi)-\varphi(\xi)\big\|_{L^2(\mathbb{R},w(\xi)d\xi)}+
\|r(\xi)\|_{L^2(\mathbb{R},w(\xi)d\xi)}+
\big\|(\mathscr{T}_{\eta}r)(\xi)\big\|_{L^2(\mathbb{R},w(\xi)d\xi)}.
\end{multline*}
Since the function \(\varphi\) is continuous and has compact support and the weight function \(w_{T}\)
is locally bounded,
\[\lim_{\eta\to0}\|(\mathscr{T}_{\eta}\varphi)(\xi)-\varphi(\xi)
\big\|_{L^2(\mathbb{R},w(\xi)d\xi)}=0.\]
According to \eqref{ENS},
\[\varlimsup_{\eta\to0}\big\|(\mathscr{T}_{\eta}r)(\xi)\big\|_{L^2(\mathbb{R},w(\xi)d\xi)}
\leq\sqrt{M_{T}(+0)}\,\varepsilon.\]
Therefore
\[\varlimsup_{\eta\to0}\|(\mathscr{T}_{\eta}x)(\xi)-
x(\xi)\big\|_{L^2(\mathbb{R},w(\xi)d\xi)}\leq
\Big(1+\sqrt{M_{T}(+0)}\Big)\varepsilon.\]
Since \(\varepsilon>0\) can be arbitrary small,
the limiting relation \eqref{StCo} holds.
\end{proof}
\subsection{On integration of
\(\boldsymbol{L^2\pmb{(}\pmb{\mathbb{R}},w\pmb{(}\xi\pmb{)}d\xi\pmb{)}}\)-valued functions.}

Let \(X(\xi,\eta)\) be a function of two variables \((\xi,\eta)\in\mathbb{R}^2\),
\(\mathscr{D}(X)\) is the domain of definition of \(X(\xi,\eta)\), \(X(\xi,\eta):\,\mathscr{D}(X)\to\mathbb{C}.\)
We assume that the function \(X(\xi,\eta)\) is defined for almost every \((\xi,\eta)\in\mathbb{R}^2\)            with respect to the two-dimensional Lebesgue measure \(d\xi{}d\eta\):
\begin{equation}
\label{aeR2}
\iint\limits_{\mathbb{R}^2\setminus{}\mathscr{D}(X)}d\xi{}d\eta=0.
\end{equation}
For \(\eta\in\mathbb{R}\), let \(\mathscr{D}_{\eta}(X)\) be the
 \textsf{\(\eta\)-section of the set \(\mathscr{D}(X)\)}:
 \begin{subequations}
 \label{seq}
 \begin{gather}
\label{etsec}
\mathscr{D}_{\eta}(X)=\lbrace{}\xi\in\mathbb{R}:\ (\xi,\eta)\in{}\mathscr{D}(X)\rbrace,
\ \  \eta \in \mathbb{R}:\\
\intertext{For \(\xi\in\mathbb{R}\), let \(\mathscr{D}_{\xi}(X)\) be the
 \(\xi\)-\textsf{section of the set }\(\mathscr{D}(X)\):}\notag\\
 \notag\\[-6.0ex]
 \label{xisec}
 \mathscr{D}_{\xi}(X)=\lbrace{}\eta\in\mathbb{R}:\ (\xi,\eta)\in{}\mathscr{D}(X)\rbrace,
 \ \  \xi \in \mathbb{R}.
\end{gather}
\end{subequations}
The sets \(\mathscr{D}_{\eta}(X)\) and \(\mathscr{D}_{\xi}(X)\) are domains of definition
of the functions \(X_{\eta}(\xi)\) and \(X_{\xi}(\eta)\) respectively.

\noindent
For \(\eta\in\mathbb{R}\), the \(\eta\)-section \(X_{\eta}(\xi)\) of the function \(X(\xi,\eta)\)
is defind as
\begin{gather}
X_{\eta}(\xi)=X(\xi,\eta), \ \ \xi\in{}\mathscr{D}_{\eta}, \ \ \eta\in\mathbb{R}.
\intertext{For \(\xi\in\mathbb{R}\), the \(\xi\)-section \(X_{\xi}(\eta)\) of the function \(X(\xi,\eta)\) is defind as}\notag\\
\notag\\[-6.0ex]
X_{\xi}(\eta)=X(\xi,\eta), \ \ \eta\in{}\mathscr{D}_{\xi}, \ \ \xi\in\mathbb{R}.
\end{gather}
For any \(\eta\in\mathbb{R}\), the set \(\mathscr{D}_{\eta}(X)\) is the domain of definition
of the function \(X_{\eta}(\xi)\).\\
From \eqref{aeR2} and Fubini's Theorem it follows that
\begin{equation}
\label{aeae}
\int\limits_{\mathbb{R}}
\bigg(\int\limits_{\mathbb{R}\setminus{}\mathscr{D}_{\eta}}d\xi\bigg)d\eta=0\,,
\qquad
\int\limits_{\mathbb{R}}
\bigg(\int\limits_{\mathbb{R}\setminus{}\mathscr{D}_{\xi}}d\eta\bigg)d\xi=0\,.
\end{equation}
First of the equalities \eqref{aeae} means that for \(d\eta\)-almost every \(\eta\in\mathbb{R}\),
the \(\eta\)-section \(X_{\eta}(\xi)\) of the function \(X(\xi,\eta)\) is defined
for \(d\xi\)-almost every \(\xi\in\mathbb{R}\). The second one means that for \(d\xi\)-almost every \(\xi\in\mathbb{R}\), the \(\xi\)-section \(X_{\xi}(\eta)\) of the function \(X(\xi,\eta)\) is defined for \(d\eta\)-almost every \(\eta\in\mathbb{R}\).

In  particular, the value \(\|X_{\eta}(\xi)\|_{L^2(\mathbb{R},w(\xi)d\xi)}\) is well defined
for \(d\eta\)-almost every \(\eta\in\mathbb{R}\):
\begin{equation}
\label{wdes}
0\leq\|X_{\eta}(\xi)\|_{L^2(\mathbb{R},w(\xi)d\xi)}\leq\infty, \ \ \textup{for}
\ d\xi-\textup{almost every} \ \xi\in\mathbb{R}.
\end{equation}

 The following Lemma is related to integration of functions
 which values are elements of the inner space product \(L^2(\mathbb{R},w(\xi)d\xi)\).
 \begin{lemma}
\label{NIIN}
Let \(w(\xi)\) be a weight function satisfying the condition \eqref{StriPos}.

Let \(X(\xi,\eta)\) be a function of two variables which is defined for
almost every \((\xi,\eta)\in\mathbb{R}^2\) with respect to two-dimensional Lebesgue
measure \(d\xi{}d\eta\), that is the condition \eqref{aeR2} holds for the domain of definition \(\mathscr{D}(X)\)
of the function \(X\).\\
We assume that\footnote{In particular, \(\|X_{\eta}(\xi)\|_{L^2(\mathbb{R},w(\xi)d\xi)}<\infty\)
 for \(d\eta\)-almost every \(\eta\in\mathbb{R}\).}
\begin{equation}
\label{FIVF}
\int\limits_{\mathbb{R}}\|X_{\eta}(\xi)\|_{L^2(\mathbb{R},w(\xi)d\xi)}d\eta<\infty
\end{equation}

Then the function \(Y(\xi)\) which is defined by the integral
\begin{equation}
\label{FDBI}
Y(\xi)=\int\limits_{\mathbb{R}}X_{\xi}(\eta)\,d\eta
\end{equation}
exists for \(d\xi\)-almost every \(\xi\in\mathbb{R}\). Moreover,
\(Y(\xi)\in{}L^2(\mathbb{R},d\xi)\) and the inequality holds
\begin{equation}
\label{NIlIN}
\|Y(\xi)\|_{L^2(\mathbb{R},w(\xi)d\xi)}\leq
\int\limits_{\mathbb{R}}\|X_{\eta}(\xi)\|_{L^2(\mathbb{R},w(\xi)d\xi)}d\eta\,.
\end{equation}
\end{lemma}
\begin{proof} Let
 \(\displaystyle{}m(\xi)=\int\limits_{\mathbb{R}}|X_{\xi}(\eta)|\,d\eta, \ \
0\leq{}m(\xi\leq\infty\).
The integral \eqref{FDBI} exists for those \(\xi\) for which \(m(\xi)<\infty\).
Let
\begin{equation}
\label{Conr}
 \|r(\xi)\|_{L^2(\mathbb{R},w(\xi)\,d\xi)}\leq1, \ \ r(\xi)\geq0.
\end{equation}
By Fubini's Theorem,
\[\int\limits_{\mathbb{R}}m(\xi)r(\xi)w(\xi)\,d\xi=
\int\limits_{\mathbb{R}}
\bigg(\int\limits_{\mathbb{R}}|X_{\eta}(\xi)|r(\xi)w(\xi)\,d\xi\bigg)\,d\eta.\]
Substituting the inequality
\[\int\limits_{\mathbb{R}}|X_{\eta}(\xi,\eta)|r(\xi)w(\xi)\,d\xi\leq
\|X_{\eta}(\xi)\|_{L^2(\mathbb{R},w(\xi)\,d\xi)},
\ \ \textup{\(\forall\,r\) \ satisfying \eqref{Conr},}
 \]
into the previous inequality, we obtain the inequality
\[\int\limits_{\mathbb{R}}m(\xi)r(\xi)w(\xi)\,d\xi\leq
\int\limits_{\mathbb{R}}\,
\|(X_{\eta}(\xi)\|_{L^2(\mathbb{R},w(\xi)d\xi)}\,d\eta\]
which holds for each \(r\) satisfying \eqref{Conr}. Taking \(\sup\) over all such \(r\),
we come to the inequality
\[\|m(\xi)\|_{L^2(\mathbb{R},w(\xi)d\xi)}\leq\int\limits_{\mathbb{R}}
\|X_{\eta}(\xi)\|_{L^2(\mathbb{R},w(\xi)d\xi)}\,d\eta.\]
If \eqref{FIVF} holds, then also \(\|m(\xi)\|_{L^2(\mathbb{R},w(\xi)d\xi)}<\infty\).
In particular, under the conditions \eqref{StriPos}, \eqref{FIVF}, the inequality \(m(\xi)<\infty\) holds
for \(d\xi\)-almost every \(\xi\in\mathbb{R}\). So the integral \eqref{FDBI} exists for
\(d\xi\)-almost every \(\xi\in\mathbb{R}\).
Since \(|Y(\xi)|\leq{}m(\xi)\), the inequality \eqref{NIlIN} holds.
\end{proof}
\subsection{Approximative Identities.}
Our proof of the completeness of Gaussians in the space \(\mathcal{H}_{w_T,w_{\Omega}}\)
is based on some approximation procedure. In this approximation procedure two families
of operators play role: the family \(\lbrace{}\mathscr{I}_{\alpha}\rbrace_{\alpha\in\mathbb{R}_{+}}\)
of integral operators and the family \(\lbrace{}\mathscr{M}_{\alpha}\rbrace_{\alpha\in\mathbb{R}_{+}}\) of multiplication operators.

The kernels of integral operators \(\lbrace{}\mathscr{I}_{\alpha}\rbrace_{\alpha\in\mathbb{R}_{+}}\)
involves the functions \(\lbrace{}G_{\alpha}(\eta)\rbrace_{\alpha\in\mathbb{R}_{+}}\):
\begin{equation}
\label{MAKe}
G_{\alpha}(\eta)=\sqrt{\alpha}\exp(-\pi\alpha\eta^2),\ \ \eta\in\mathbb{R}.
\end{equation}
The functions \(G_{\alpha}(\eta)\) posses the properties:
\begin{subequations}
\label{PropG}
\begin{enumerate}
\item{\ }\\[-6.0ex]
\begin{equation}
\label{Pos}
G_{\alpha}(\eta)>0, \ \ \forall\,\eta\in\mathbb{R},\,\forall\,\alpha\in\mathbb{R}_{+}.
\end{equation}
\item {\ }\\[-6.0ex]
\begin{equation}
\label{Norm}
\int\limits_{\mathbb{R}}G_{\alpha}(\eta)\,d\eta=1, \\ \ \forall\,\alpha\in\mathbb{R}_{+}.
\end{equation}
\item For each \(\mu\in\mathbb{R}_{+}\) and \(\delta\in\mathbb{R}_{+}\),
\begin{equation}
\label{conc}
\lim_{\alpha\to\infty}\int\limits_{\mathbb{R}\setminus(-\delta,\delta)}
G_{\alpha}(\eta)\exp(\mu|\eta|)d\eta=0.
\end{equation}
\end{enumerate}
\end{subequations}

\begin{definition}
\label{DApU}
Let \(\alpha\in\mathbb{R}_{+}\).
\begin{enumerate}
\item[\textup{1}.]
The operator \(\mathscr{I}_{\alpha}:%
L^2(\mathbb{R},w(\xi)d\xi)\to{}L^2(\mathbb{R},w(\xi)d\xi)\) is defined as
\begin{equation}
\label{AppUnO}
(\mathscr{I}_{\alpha}x)(\xi)=
\int\limits_{\mathbb{R}}G_{\alpha}(\eta)(\mathscr{T}_{\eta}x)(\xi)d\xi,
\end{equation}
where \(\mathscr{T}_{\eta}\) is the shift operator defined by \eqref{DefShf}.
\item[\textup{2}.]
The operator \(\mathscr{M}_{\alpha}:%
L^2(\mathbb{R},w(\xi)d\xi)\to{}L^2(\mathbb{R},w(\xi)d\xi)\) is defined as
\begin{equation}
\label{AppMuO}
(\mathscr{M}_{\alpha}x)(\xi)=\exp\big(-\tfrac{\pi}{\alpha}\,{\textstyle\xi^2}\big)\,x(\xi).
\end{equation}
\end{enumerate}
\end{definition}
\begin{lemma}
\label{AppIdL}
Let \(w(\xi)\) be a weight function which satisfy the regularity condition\footnote{%
See Definition \ref{DefReCo}.
}.\\ %
Then
\begin{enumerate}
\item[\textup{1.}]
For each \(\alpha>0\), the operator \(\mathscr{I}_{\alpha}\) is \, a bounded operator \,
in the space \(L^2(\mathbb{R},w(\xi)d\xi)\),
and the estimate
\begin{equation}
\label{IFN}
\|\mathscr{I}_{\alpha}\|_{L^2(\mathbb{R},w(\xi)d\xi)
\to{L^2(\mathbb{R},w(\xi)d\xi)}}\leq
\int\limits_{\mathbb{R}}G_{\alpha}(\eta)\sqrt{M_{w}(|\eta|)}\,d\eta, \ \ \forall\,\alpha\in\mathbb{R}^{+},
\end{equation}
holds, where \(M_{w}\) is the multiplicative modulus of continuity\footnote{%
See Definition \ref{DMMC}.
}%
of the weight function \(w\). In particular,
\begin{equation}
\label{IFNb}
\|\mathscr{I}_{\alpha}\|_{L^2(\mathbb{R},w(\xi)d\xi)
\to{L^2(\mathbb{R},w(\xi)d\xi)}}\leq
\sqrt{C_{w}}\int\limits_{\mathbb{R}}G_{\alpha}(\eta)\exp(\mu_{w}|\eta|/2)\,d\eta, \ \ \forall\,\alpha\in\mathbb{R}^{+},
\end{equation}
where \(C_{w}\) and \(\mu_{w}\) are the same that in \eqref{MCE}.
\item[\textup{2}.] The family \(\lbrace\mathscr{I}_{\alpha}\rbrace_{\alpha\in\mathbb{R}_{+}}\)
is an approximative identity in the space \(L^2(\mathbb{R},w(\xi)d\xi)\):
\begin{equation}
\label{ApprUn}
\lim_{\alpha\to\infty}
\big\|(\mathscr{I}_{\alpha}x)(\xi)-x(\xi)\big\|_{L^2(\mathbb{R},w(\xi)d\xi)}=0,
\ \ \textup{for each} \ x\in{}L^2(\mathbb{R},w(\xi)d\xi).
\end{equation}
\end{enumerate}
\end{lemma}
\begin{proof} {\ }\\
\textsf{1.} We apply Lemma \ref{NIIN} to the function
\(X(\xi,\eta)=G_{\alpha}(\eta)(\mathscr{T}_{\eta}x)(\xi)\). According to Lemma
\ref{PrTrO}, the inequality \eqref{ENS} holds. This inequality implies the estimate for
\(L^2(\mathbb{R},w(\xi)d\xi)\)-norm of \(\eta\)-sections \(X_{\eta}(\xi)\) of the function \(X\):
\[\|X_{\eta}(\xi)\|_{L^2(\mathbb{R},w(\xi)d\xi)}\leq{}G_{\alpha}(\eta)\sqrt{M_{w}(\eta)}
\|x(\xi)\|_{L^2(\mathbb{R},w(\xi)d\xi)}.\]
The inequality \eqref{IFN} is a consequence of the last inequality. Using the estimate
\eqref{MCE} for \(M_{w}\), we come to the inequality \eqref{IFNb}.\\
\textsf{2.}  In view of \eqref{Norm},
\[(\mathscr{I}_{\alpha}x)(\xi)-x(\xi)=
\int\limits_{\mathbb{R}}G_{\alpha}(\eta)\big((\mathscr{T}_{\eta}x)(\xi)-x(\xi)\big)d\xi.\]
Let us fix  \(x(\xi)\in{}L^2(\mathbb{R},w(\xi)d\xi)\).
According to Lemma \ref{NIIN},
\begin{equation}
\label{PreEst}
\|(\mathscr{I}_{\alpha}x)(\xi)-x(\xi)\|_{L^2(\mathbb{R},w(\xi)d\xi)}\leq
\int\limits_{\mathbb{R}}G_{\alpha}(\eta)
\|(\mathscr{T}_{\eta}x)(\xi)-x(\xi)\|_{L^2(\mathbb{R},w(\xi)d\xi)}d\eta.
\end{equation}
By statement 2 of Lemma \ref{PrTrO}, for any \(\varepsilon>0\) there exists \(\delta>0\),
\(\delta=\delta(\varepsilon,x)\), such that
\begin{equation}
\label{ShEs}
\|(\mathscr{T}_{\eta}x)(\xi)-x(\xi)\|_{L^2(\mathbb{R},w(\xi)d\xi)}<\varepsilon
\ \ \textup{if} \ \ |\eta|<\delta.
\end{equation}
Splitting the integral in the right hand side of \eqref{PreEst} into the sum of integrals
taken over \((-\delta,\delta)\) and \(\mathbb{R}\setminus(-\delta,\delta)\), we obtain
\begin{multline}
\label{abc}
\|(\mathscr{I}_{\alpha}x)(\xi)-x(\xi)\|_{L^2(\mathbb{R},w(\xi)d\xi)}\leq
\!\int\limits_{(-\delta,\delta)}G_{\alpha}(\eta)
\|(\mathscr{T}_{\eta}x)(\xi)-x(\xi)\|_{L^2(\mathbb{R},w(\xi)d\xi)}d\eta\\
\hfill+
\int\limits_{\mathbb{R}\setminus(-\delta,\delta)}G_{\alpha}(\eta)
\Big(\|(\mathscr{T}_{\eta}x)(\xi)\|_{L^2(\mathbb{R},w(\xi)d\xi)}
+\|x(\xi)\|_{L^2(\mathbb{R},w(\xi)d\xi)}\Big)d\eta.
\end{multline}
From \eqref{abc} and from the inequalities \eqref{ShEs}, \eqref{ENS}, \eqref{MCE}
we obtain that
\begin{multline}
\label{ABC}
\|(\mathscr{I}_{\alpha}x)(\xi)-x(\xi)\|_{L^2(\mathbb{R},w(\xi)d\xi)}\leq\\
\varepsilon\int\limits_{(-\delta,\delta)}G_{\alpha}(\eta)d\eta+\big(\sqrt{C_{w}}+1\big)\!\!\!\!
\int\limits_{\mathbb{R}\setminus(-\delta,\delta)}G_{\alpha}(\eta)
\exp\big(\tfrac{\mu_{w}}{2}|\eta|\big)d\eta
\,\,\,\,
\|x\|_{L^2(\mathbb{R},w(\xi)d\xi)}.
\end{multline}
From \eqref{ABC} and from the properties \eqref{PropG} of the functional family
\(\lbrace{}G_{\alpha}\rbrace_{\alpha\in\mathbb{R}_{+}}\) it follows that
\[\varlimsup_{\alpha\to\infty}
\|(\mathscr{I}_{\alpha}x)(\xi)-x(\xi)\|_{L^2(\mathbb{R},w(\xi)d\xi)}\leq\varepsilon.\]
Since \(\varepsilon>0\) is arbitrary, \eqref{ApprUn} holds.
\end{proof}
\begin{lemma}
\label{ApIdWN}
Let \(w(\xi)\) be a weight function which satisfy the regularity condition.

Then each of the two families
\(\lbrace\mathscr{M}_{\alpha}\mathscr{I}_{\alpha}\rbrace_{\alpha\in\mathbb{R}_{+}}\) and
\(\lbrace\mathscr{I}_{\alpha}\mathscr{M}_{\alpha}\rbrace_{\alpha\in\mathbb{R}_{+}}\)
of operators
is an approximative identity in the space \(L^2(\mathbb{R},w(\xi)d\xi)\):
\begin{subequations}
\label{ApIdwn}
\begin{gather}
\label{ApIdwn1}
\lim_{\alpha\to\infty}
\big\|(\mathscr{I}_{\alpha}\mathscr{M}_{\alpha}x)
(\xi)-x(\xi)\big\|_{L^2(\mathbb{R},w(\xi)d\xi)}=0,
\ \ \textup{for each} \ x\in{}L^2(\mathbb{R},w(\xi)d\xi),\\
\label{ApIdwn2}
\lim_{\alpha\to\infty}
\big\|(\mathscr{M}_{\alpha}\mathscr{I}_{\alpha}x)(\xi)-x(\xi)\big\|_{L^2(\mathbb{R},w(\xi)d\xi)}=0,
\ \ \textup{for each} \ x\in{}L^2(\mathbb{R},w(\xi)d\xi).
\end{gather}
\end{subequations}
\end{lemma}
\begin{proof}
From \eqref{IFNb} it follows that the family of operators
\(\lbrace\mathscr{I}_{\alpha}\rbrace_{\alpha\in[1,\infty)}\) is uniformly bounded:
\begin{equation}
\label{UniBou}
\sup_{\alpha\in[1,\infty)}
\|\mathscr{I}_{\alpha}\|_{L^2(\mathbb{R},w(\xi)d\xi)\to{}L^2(\mathbb{R},w(\xi)d\xi)}<\infty.
\end{equation}
The operator \(\mathscr{M}_{\alpha}\) is contractive for any \(\alpha\in\mathbb{R}_{+}\):
\[\|\mathscr{M}_{\alpha}\|_{L^2(\mathbb{R},w(\xi)d\xi\to{}L^2(\mathbb{R},w(\xi)d\xi}=1,
\ \ \forall\,\alpha>0.
\]
In particular, the family \(\lbrace\mathscr{M}_{\alpha}\mathscr{I}_{\alpha}\rbrace_{\alpha\in\mathbb{R}_{+}}\)
is uniformly bounded.

In Lemma \ref{AppIdL} we established that
the family \(\lbrace\mathscr{I}_{\alpha}\rbrace_{\alpha\in\mathbb{R}_{+}}\)
is an approximative identity in the space \(L^2(\mathbb{R},w(\xi)d\xi)\):
\[\lim_{\alpha\to\infty}\mathscr{I}_{\alpha}=\mathscr{I},\]
where \(\mathscr{I}\) is the identity operator in the space \(L^2(\mathbb{R},w(\xi)d\xi)\)
and  convergence is the strong convergence in this space. It is clear that
the family \(\lbrace\mathscr{M}_{\alpha}\rbrace_{\alpha\in\mathbb{R}_{+}}\) also is
an approximative identity in the space \(L^2(\mathbb{R},w(\xi)d\xi)\):
\[\lim_{\alpha\to\infty}\mathscr{M}_{\alpha}=\mathscr{I},\]
where convergence is the strong convergence in \(L^2(\mathbb{R},w(\xi)d\xi)\).
The assertion of Lemma \ref{ApIdWN} follows now from the equalities
\begin{align*}
\mathscr{I}_{\alpha}\mathscr{M}_{\alpha}-\mathscr{I}&=
\mathscr{I}_{\alpha}(\mathscr{M}_{\alpha}-\mathscr{I})+(\mathscr{I}_{\alpha}-\mathscr{I}),\\
\mathscr{M}_{\alpha}\mathscr{I}_{\alpha}-\mathscr{I}&=
\mathscr{M}_{\alpha}(\mathscr{I}_{\alpha}-\mathscr{I})+(\mathscr{M}_{\alpha}-\mathscr{I}).
\end{align*}
\end{proof}

\section{The completeness of a system of Gaussians\\ in the space \(\boldsymbol{\mathcal{H}_{w_T,w_{\Omega}}}\).}
\begin{definition}
\label{DeGa}
For each \(\alpha\in\mathbb{R}_{+}\) and \(\tau\in\mathbb{R}\), \textsf{the Gaussian} \(\boldsymbol{g_{\alpha,\tau}}\)
 is a function
\begin{equation}
\label{Ga}
g_{\alpha,\tau}(t)=\sqrt{\alpha}\exp\big(-\pi\alpha(t-\tau)^2)\big)
\end{equation}
 of the variable \(t\in\mathbb{R}\). The Gaussians form a two-parametric family
 \(\lbrace{}g_{\alpha,\tau}\rbrace_{\alpha,\tau}\) of functions on \(\mathbb{R}\) which is parametrized
 by the parameters \(\alpha\in\mathbb{R}_{+}\) and \(\tau\in\mathbb{R}\).
\end{definition}

\noindent
The Fourier transform \(\widehat{g_{\alpha,\tau}}(\omega)\) of the Gaussian \(g_{\alpha,\tau}(t)\) is
\begin{equation}
\label{FTGa}
\widehat{g_{\alpha,\tau}}(\omega)=
\exp\big(-\tfrac{\pi}{\alpha}\omega^2-2\pi{}i\tau\omega\big).
\end{equation}
\begin{lemma}[I.Schur]
\label{ScTe}
For  a function \(K(t,\tau),\,t\in\mathbb{R},\,\tau\in\mathbb{R}\) of two variables, let as define
the values
\begin{equation}
\label{SchVal}
N_1(K)=\sup_{\tau\in\mathbb{R}}\int\limits_{\mathbb{R}}|K(t,\tau)|dt,\quad
N_\infty(K)=\sup_{t\in\mathbb{R}}\int\limits_{\mathbb{R}}|K(t,\tau)|d\tau.
\end{equation}
If the condition 
\begin{gather}
\label{SchCon}
N_{1}(K)<\infty,\quad{}N_{\infty}(K)<\infty
\end{gather}
are fulfilled, then for each \(f(t)\in{}L^2(\mathbb{R},dt)\), \(g(\tau)\in{}L^2(\mathbb{R},d\tau)\)
the double integral
\begin{equation}
\label{DoIn}
\iint\limits_{\mathbb{R}\times\mathbb{R}}f(t)K(t,\tau)\overline{g(\tau)}\,dtd\tau
\end{equation}
exists and admits the estimate
\begin{equation}
\label{DoInEs}
\bigg|\iint\limits_{\mathbb{R}\times\mathbb{R}}f(t)K(t,\tau)\overline{g(\tau)}\,dtd\tau\bigg|
\leq\sqrt{N_1(K)N_\infty(K)}\,\,\big\|f\big\|_{L^2(\mathbb{R},dt)}\,\big\|g\big\|_{L^2(\mathbb{R},d\tau)}.
\end{equation}
\end{lemma}
\begin{proof}
The integral \eqref{DoIn} exists if \
\(\displaystyle\iint\limits_{\mathbb{R}\times\mathbb{R}}|f(t)|\,|K(t,\tau)|\,|g(\tau)|\,dtd\tau<\infty.\) \ %
Applying Cauchy-Schwarz inequality for double integral and Fubini Theorem, we obtain:
\begin{gather*}
\iint\limits_{\mathbb{R}\times\mathbb{R}}|f(t)|K(t,\tau)|g(\tau)|\,dtd\tau=
\iint\limits_{\mathbb{R}\times\mathbb{R}}\Big(|f(t)|K(t,\tau)^{1/2}\Big)
\Big(K(t,\tau)^{1/2}|g(\tau)|\Big)\,dtd\tau\\
\leq\Big(\iint\limits_{\mathbb{R}\times\mathbb{R}}\Big(|f(t)|^2K(t,\tau)dt\,d\tau\Big)^{\frac{1}{2}}
\cdot\Big(\iint\limits_{\mathbb{R}\times\mathbb{R}}\Big(|g(\tau)|^2K(t,\tau)dt\,d\tau\Big)^{\frac{1}{2}}\\
=\Bigg(\int\limits_{\mathbb{R}}
\bigg(\int\limits_{\mathbb{R}}K(t,\tau)\,d\tau\bigg)|f(t)|^2dt\Bigg)^{\frac{1}{2}}
\cdot\Bigg(\int\limits_{\mathbb{R}}
\bigg(\int\limits_{\mathbb{R}}K(t,\tau)\,dt\bigg)|g(\tau)|^2d\tau\Bigg)^{\frac{1}{2}}\\
\hspace*{0.6\linewidth}\leq\sqrt{N_\infty(K)N_1(K)}\,\big\|f\big\|^2_{L^2(\mathbb{R},dt)},
\end{gather*}
\end{proof}
\begin{theorem}
\label{GaIn}
We assume that:
\begin{enumerate}
\item[\textup{1.}]
 The weight functions \(w_{T}\) and \(w_{\Omega}\) satisfy
the non-degeneracy conditions,
\textup{(Definition \ref{DeCo})};
\item[\textup{2.}]
The weight functions \(1+w_T\) and \(1+w_{\Omega}\) satisfy the regularity conditions, \textup{(Definition \ref{DefReCo})}.\\
\end{enumerate}
Then
\begin{enumerate}
\item[\textup{1}.]
Each Gaussian \(g_{\alpha,\tau}\) belongs to the space \(\mathcal{H}_{w_T,w_{\Omega}}\).
\item[\textup{2}.]
The functional family \(\lbrace{}g_{\alpha,\tau}\rbrace_{\alpha,\tau}\), where \(\alpha\) runs
over \(\mathbb{R}_{+}\) and \(\tau\) runs over \(\mathbb{R}\), is complete\footnote{%
That is the linear hall of this family is dense in \(\mathcal{H}_{w_T,w_{\Omega}}\).
} %
in the space
\(\mathcal{H}_{w_{T},w_{\Omega}}\).
\end{enumerate}
\end{theorem}
\begin{proof} {\ }\\
\textbf{1}.
From  \eqref{Ga}, \eqref{FTGa} and \eqref{IWFE} it is evident that
\begin{equation*}
\int\limits_{\mathbb{R}}\big|g_{\alpha,\tau}(t)\big|^2w_{T}(t)dt+
\int\limits_{\mathbb{R}}\big|\widehat{g_{\alpha,\tau}}(\omega)\big|^2w_{\Omega}(\omega)d\omega<\infty
\ \ \forall\,\alpha\in\mathbb{R}_{+},\,\ \forall\,\tau\in\mathbb{R}.
\end{equation*}
\textbf{2}. \textsf{Step 1}. According to Theorem \ref{Com}, the inner product space \(\mathcal{H}_{w_{T},w_{\Omega}}\)
is compete. Given \(f\in\mathcal{H}_{w_{T},w_{\Omega}}\),  we have to prove that from the orthogonality condition
\begin{equation}
\label{OrCo}
\langle{}f,g_{\alpha,\tau}\rangle_{\mathcal{H}_{w_{T},w_{\Omega}}}=0, \ \ \forall\,\alpha\in\mathbb{R}_{+}, \ \forall\,\tau\in\mathbb{R}
\end{equation}
it follows that \(\langle{}f,f\rangle_{\mathcal{H}_{w_{T},w_{\Omega}}}=0\). According to \eqref{InPr}, \eqref{Ga}, \eqref{FTGa}, the condition \eqref{OrCo} can be presented in the form
\begin{multline}
\label{OrCo1}
\hfil
\int\limits_{\mathbb{R}}f(t)\sqrt{\alpha}e^{-\pi\alpha(t-\tau)^2}w_{T}(t)dt+
\int\limits_{\mathbb{R}}\hat{f}(\omega)e^{-\frac{\pi}{\alpha}\omega^2}
e^{2\pi{}i\tau\omega}w_{\Omega}(\omega)d\omega=0, \\
\forall\,\alpha\in\mathbb{R}_{+}, \ \forall\,\tau\in\mathbb{R}.
\end{multline}
Is appropriate to recall\footnote{%
See Corollary \ref{Cor}.
} %
 that since \(f\in\mathcal{H}_{w_T,w_{\Omega}}\),
\begin{subequations}
\label{ATR}
\begin{gather}
\label{ATRa}
f(t)\in{}L^2(\mathbb{R},(1+w_T(t))dt),\\
\label{ATRb}
\hat{f}(\omega)\in{}L^2(\mathbb{R},(1+w_{\Omega}(\omega))d\omega).
\end{gather}
\end{subequations}
 We multiply the equality \eqref{OrCo1} with \(e^{-\frac{\pi}{\alpha}\tau^2}\overline{f(\tau)}\)
and integrate by the measure \(d\tau\) over \(\mathbb{R}\).
\begin{gather}
\int\limits_{\mathbb{R}}\bigg(\int\limits_{\mathbb{R}}f(t)\sqrt{\alpha}%
e^{-\pi\alpha(t-\tau)^2}w_{T}(t)dt\bigg)
e^{-\frac{\pi}{\alpha}\tau^2}\overline{f(\tau)}d\tau\notag\\[-1.0ex]
\label{MulIn}\\[-1.0ex]
+\int\limits_{\mathbb{R}}\bigg(\int\limits_{\mathbb{R}}
\hat{f}(\omega)e^{-\frac{\pi}{\alpha}\omega^2}
e^{2\pi{}i\tau\omega}w_{\Omega}(\omega)d\omega\bigg)
e^{-\frac{\pi}{\alpha}\tau^2}\overline{f(\tau)}d\tau=0.
\notag
\end{gather}
\textbf{Step 2}. Let us proof the existence of the iterated integral
 \begin{equation}
 \label{ItInt1}
 \int\limits_{\mathbb{R}}\bigg(\int\limits_{\mathbb{R}}f(t)\sqrt{\alpha}%
e^{-\pi\alpha(t-\tau)^2}w_{T}(t)dt\bigg)
e^{-\frac{\pi}{\alpha}\tau^2}\overline{f(\tau)}d\tau
\end{equation}
which appears as the first summand of \eqref{MulIn}.
Using the inequality
\[w_{T}(t)\leq{}C_{T}e^{\mu_{T}|t-\tau|}e^{\mu_{T}|\tau|}, \ \forall\,t\in\mathbb{R},\,\tau\in\mathbb{R},\]
which follows from \eqref{IWFEa}, we obtain the inequality
\begin{equation}
\label{Ker}
\iint\limits_{\mathbb{R}\times\mathbb{R}}|f(t)|
\sqrt{\alpha}e^{-\pi\alpha(t-\tau)^2}w_{T}(t)e^{-\frac{\pi}{\alpha}\tau^2}|f(\tau)|\,dtd\tau
\leq\iint\limits_{\mathbb{R}\times\mathbb{R}}|f(t)|K_{T}(t,\tau)|f(\tau)|\,dtd\tau,
\end{equation}
where the kernel \(K_{T}\) is of the form
\begin{equation}
\label{K1}
K_{T}(t,\tau)=C_{T}\sqrt{\alpha}\exp\big(-\pi\alpha(t-\tau)^2+\mu_{T}|t-\tau|\big)
\exp\big(-\tfrac{\pi}{\alpha}\tau^2+\mu_{T}|\tau|\big).
\end{equation}
(The constants \(C_{T}\) and \(\mu_{T}\) are the same that in \eqref{IWFEa}.)
By direct calculation,
\begin{equation*}
\max_{\tau\in\mathbb{R}}\exp\big(-\tfrac{\pi}{\alpha}\tau^2+\mu_{T}|\tau|\big)=
\exp\big(\tfrac{\mu_{T}^2\alpha}{4\pi}\big).
\end{equation*}
So, the kernel \(K_{T}(t,\tau)\) admits estimate
\[0\leq
K_{T}(t,\tau)\leq{}C_T\sqrt{\alpha}\exp\big(\tfrac{\mu_{T}^2\alpha}{4\pi}\big)
\exp\big(-\pi\alpha(t-\tau)^2+\mu_{T}|t-\tau|\big), \ \ \forall\,t\in\mathbb{R},\,\tau\in\mathbb{R}.
\]
From the last inequality, the estimates
\[\sup_{\tau\in\mathbb{R}}\int\limits_{\mathbb{R}}K_{T}(t,\tau)dt\leq{}I(\alpha), \quad \sup_{t\in\mathbb{R}}\int\limits_{\mathbb{R}}K_{T}(t,\tau)dt\leq{}I(\alpha)\]
hold, where
\[I(\alpha)=C_T\sqrt{\alpha}\exp\big(\tfrac{\mu_{T}^2\alpha}{4\pi}\big)
\int\limits_{\mathbb{R}}\exp\big(-\pi\alpha\xi^2+\mu_{T}|\xi|\big)\,d\xi,
\ \ I(\alpha)<\infty \ \forall\,\alpha\in\mathbb{R}_{+}.\]

Applying Lemma \ref{ScTe} to the kernel \(K_{T}(t,\tau)\), \eqref{K1}, we come
to the inequality
\[\iint\limits_{\mathbb{R}\times\mathbb{R}}|f(t)|K_{T}(t,\tau)|f(\tau)|\,dtd\tau
\leq I(\alpha)\big\|f\big\|^2_{L^2(\mathbb(R),dt)}<\infty.\]

Therefore the double integral in the left hand side of \eqref{Ker} is finite for every \(\alpha\in\mathbb{R}_{+}\)
and for every \(f\in{}L^2(\mathbb{R},dt)\), in particular, for every \(f\in\mathcal{H}_{w_{T},w_{\Omega}}\).

By Fubini's Theorem, the iterated integral \eqref{ItInt1} exists. Moreover we can
 interchange the order of
integration in this integral:
\begin{multline*}
\int\limits_{\mathbb{R}}\bigg(\int\limits_{\mathbb{R}}
f(t)\,\sqrt{\alpha}e^{-\pi\alpha(t-\tau)^2}w_{T}(t)dt\bigg)
e^{-\frac{\pi}{\alpha}\tau^2}\overline{f(\tau)}d\tau\\=
\int\limits_{\mathbb{R}}f(t)\bigg(\int\limits_{\mathbb{R}}
\sqrt{\alpha}e^{-\pi\alpha(t-\tau)^2}
e^{-\frac{\pi}{\alpha}\tau^2}\overline{f(\tau)}d\tau\bigg)w_{T}(t)dt.
\end{multline*}
Changing the variable \(\tau\to{}t-\tau\) in the inner integral of the right hand side
of the previous equality, we come to the equality
\begin{multline*}
\int\limits_{\mathbb{R}}\bigg(\int\limits_{\mathbb{R}}
f(t)\,\sqrt{\alpha}e^{-\pi\alpha(t-\tau)^2}w_{T}(t)dt\bigg)
e^{-\frac{\pi}{\alpha}\tau^2}\overline{f(\tau)}d\tau\\=
\int\limits_{\mathbb{R}}f(t)\bigg(\int\limits_{\mathbb{R}}
\sqrt{\alpha}e^{-\pi\alpha\tau^2}
e^{-\frac{\pi}{\alpha}(t-\tau)^2}\overline{f(t-\tau)}d\tau\bigg)w_{T}(t)dt.
\end{multline*}
In other words,
\begin{equation}
\label{ChOrd}
\int\limits_{\mathbb{R}}\bigg(\int\limits_{\mathbb{R}}
f(t)\sqrt{\alpha}e^{-\pi\alpha(t-\tau)^2}w_{T}(t)dt\bigg)
e^{-\frac{\pi}{\alpha}\tau^2}\overline{f(\tau)}d\tau=
\langle{}f,\mathfrak{I}_{\alpha}\mathscr{M}_{\alpha}f\rangle_{L^2(\mathbb{R},w_{T}(t)dt)},
\end{equation}
where \(\mathscr{I}_{\alpha}\) and \(\mathscr{M}_{\alpha}\) are the operators which were
introduced in Definition \ref{DApU}.

Since \(f\in\mathcal{H}_{w_T,w_{\omega}}\), the inclusion \eqref{ATRa} holds.
According to Lemma \ref{ApIdWN}, the family
\(\lbrace\mathscr{I}_{\alpha}\mathscr{M}_{\alpha}\rbrace_{\alpha\to\infty}\)
is an approximative identity
in \(L^2(\mathbb{R},(1+w_T(t))dt)\):
\[
\lim_{\alpha\to\infty}\big\|(\mathscr{I}_{\alpha}\mathscr{M}_{\alpha}x)(t)
-x(t)
\big\|_{L^2(\mathbb{R},(1+w_{T}(t))dt)}=0, \ \ \forall\,x\in{}L^2(\mathbb{R},(1+w_{T}(t))dt).
\]

All the more,
\[
\lim_{\alpha\to\infty}\big\|(\mathscr{I}_{\alpha}\mathscr{M}_{\alpha}f)(t)
-f(t)
\big\|_{L^2(\mathbb{R},w_{T}(t)dt)}=0, \ \ \textup{for} \ f \ \textup{which appears in} \
\eqref{ChOrd}.
\]
Comparing the last equality with \eqref{ChOrd}, we conclude that
\begin{equation}
\label{LS1}
\lim_{\alpha\to\infty}\int\limits_{\mathbb{R}}
\bigg(\int\limits_{\mathbb{R}}f(t)\tfrac{1}{\sqrt{\alpha}}
e^{-\pi\alpha(t-\tau)^2}w_{T}(t)dt\bigg)
e^{-\frac{\pi}{\alpha}\tau^2}\overline{f(\tau)}d\tau=
\langle{}f,f\rangle_{L^2(\mathbb{R},w_{T}(t)dt)}.
\end{equation}
\textsf{Step3}. Let us elaborate the second summand in \eqref{MulIn}, i.e. the expression
\begin{equation}
\label{ItInt2}
\int\limits_{\mathbb{R}} \bigg(\int\limits_{\mathbb{R}}\hat{f}(\omega)e^{-\frac{\pi}{\alpha}\omega^2}
e^{2\pi{}i\tau\omega}w_{\Omega}(\omega)d\omega\bigg)
e^{-\frac{\pi}{\alpha}\tau^2}\overline{f(\tau)}d\tau.
\end{equation}
Using the inequality
\[w_{\Omega}(\omega)\leq{}C_{\Omega}e^{\mu_{\Omega}|\omega|}, \ \forall\,\omega\in\mathbb{R},\]
which is the inequality \eqref{IWFEa}, we obtain
\begin{equation}
\label{Ker2}
\iint\limits_{\mathbb{R}\times\mathbb{R}}|\hat{f}(\omega)|
e^{-\frac{\pi}{\alpha}\omega^2}w_{\Omega}(\omega)
e^{-\frac{\pi}{\alpha}\tau^2}|f(\tau)|\,d\omega{}d\tau
\leq\iint\limits_{\mathbb{R}\times\mathbb{R}}
|\hat{f}(\omega)|K_{\Omega}(\omega,\tau)|f(\tau)|\,d\omega{}d\tau,
\end{equation}
where the kernel \(K_{\Omega}\) is of the form
\begin{equation}
\label{K2}
K_{\Omega}(\omega,\tau)=C_{\Omega}\exp\big(-\tfrac{\pi}{\alpha}\omega^2+\mu_{\Omega}|\omega|\big)
\exp\big(-\tfrac{\pi}{\alpha}\tau^2\big).
\end{equation}
Since
\[\max_{\omega\in\mathbb{R}}\exp\big(-\tfrac{\pi}{\alpha}\omega^2+\mu_{\Omega}|\omega|\big)
=\exp\big(\tfrac{\mu_{\Omega}^2\alpha}{4\pi}\big), \qquad\max_{\tau\in\mathbb{R}}
\exp\big(-\tfrac{\pi}{\alpha}\tau^2\big)=1,\]
the kernel \(K_{\Omega}(\omega,\tau)\), \eqref{K2}, admits two estimates:
\begin{subequations}
\label{EstK2}
\begin{align}
\label{EstK21}
K_{\Omega}(\omega,\tau)\leq{}C_{\Omega}\exp\big(-\tfrac{\pi}{\alpha}\omega^2+\mu_{\Omega}|\omega|\big),
\ \ &\forall\,\omega\in\mathbb{R},\,\forall\,\tau\in\mathbb{R},\\[-3.5ex]
\intertext{and}
\notag\\[-7.5ex]
\label{EstK22}
K_{\Omega}(\omega,\tau)\leq{}C_{\Omega}\exp\big(\tfrac{\mu_{\Omega}^2\alpha}{4\pi}\big)
\exp\big(-\tfrac{\pi}{\alpha}\tau^2\big),
\ \ &\forall\,\omega\in\mathbb{R},\,\forall\,\tau\in\mathbb{R}.
\end{align}
\end{subequations}
From \eqref{EstK21} it follows that
\begin{equation*}
\int\limits_{\mathbb{R}}K_{\Omega}(\omega,\tau)\,d\omega\leq{}I_1(\alpha), \ \ %
\forall\,\tau\in\mathbb{R},
\end{equation*}
where
\[
I_1(\alpha)=C_{\Omega}\int\limits_{\mathbb{R}}
\exp\big(-\tfrac{\pi}{\alpha}\omega^2+\mu_{\Omega}|\omega|\big)\,d\omega<\infty, \ \
\forall\,\alpha\in\mathbb{R}_{+}.
\]
From \eqref{EstK22} it follows that
\begin{equation*}
\int\limits_{\mathbb{R}}K_{\Omega}(\omega,\tau)\,d\tau\leq{}I_\infty(\alpha), \ \ %
\forall\,\omega\in\mathbb{R},
\end{equation*}
where
\[
I_\infty(\alpha)=C_{\Omega}\exp\big(\tfrac{\mu_{\Omega}^2\alpha}{4\pi}\big)\int\limits_{\mathbb{R}}
\exp\big(-\tfrac{\pi}{\alpha}\tau^2\big)\,d\tau<\infty, \ \
\forall\,\alpha\in\mathbb{R}_{+}.
\]
Applying Lemma \ref{ScTe} to the kernel \(K_{\Omega}(\omega,\tau)\), \eqref{K2}, we come
to the inequality
\[\iint\limits_{\mathbb{R}\times\mathbb{R}}|\hat{f}(\omega)|K_{\Omega}(\omega,\tau)|f(\tau)|\,dtd\tau
\leq \sqrt{I_1(\alpha)I_{\infty}(\alpha)}
\big\|\hat{f}\big\|_{L^2(\mathbb{R},d\omega)}
\big\|f\big\|_{L^2(\mathbb{R},d\tau)}<\infty.\]
Therefore the double integral in the left hand side of \eqref{Ker2} is finite for every \(\alpha\in\mathbb{R}_{+}\)
and for every \(\hat{f}\in{}L^2(\mathbb{R},d\omega)\), i.e. \(f\in{}L^2(\mathbb{R},d\tau)\), in particular, for every \(f\in\mathcal{H}_{w_{T},w_{\Omega}}\).

By Fubini's Theorem, the iterated integral \eqref{ItInt2} exists. Moreover we can
interchange the order of integration in this integral:
\begin{multline}
\label{ChOrd2}
\int\limits_{\mathbb{R}} \bigg(\int\limits_{\mathbb{R}}\hat{f}(\omega)e^{-\frac{\pi}{\alpha}\omega^2}
e^{2\pi{}i\tau\omega}w_{\Omega}(\omega)d\omega\bigg)
e^{-\frac{\pi}{\alpha}\tau^2}\overline{f(\tau)}d\tau\\=
\int\limits_{\mathbb{R}}\hat{f}(\omega)e^{-\frac{\pi}{\alpha}\omega^2}
\bigg(\int\limits_{\mathbb{R}}
e^{-\frac{\pi}{\alpha}\tau^2+2\pi{}i\omega\tau}\overline{f(\tau)}d\tau\bigg)w_{\Omega}(\omega)d\omega.
\end{multline}
The Fourier Transform of the exponential \(e^{-\frac{\pi}{\alpha}\tau^2+2\pi{}i\omega\tau}\),
considered as a function of the variable \(\tau\), is
\[\int\limits_{\mathbb{R}}e^{-\frac{\pi}{\alpha}\tau^2+2\pi{}i\omega\tau}
e^{-2\pi{}i\lambda\tau}d\tau=\sqrt{\alpha}e^{-\pi\alpha(\omega-\lambda)^2}.\]
In view of the Parseval equality,
\[\int\limits_{\mathbb{R}}e^{-\frac{\pi}{\alpha}\tau^2+2\pi{}i\omega\tau}
\overline{f(\tau)}d\tau=
\int\limits_{\mathbb{R}}\sqrt{\alpha}e^{-\pi\alpha(\omega-\lambda)^2}
\overline{\hat{f}(\lambda)}\,d\lambda.
\]
Thus the iterated integral in the right hand side of \eqref{ChOrd2} takes the form
\begin{multline*}
\int\limits_{\mathbb{R}}\hat{f}(\omega)e^{-\frac{\pi}{\alpha}\omega^2}
\bigg(\int\limits_{\mathbb{R}}
e^{-\frac{\pi}{\alpha}\tau^2+2\pi{}i\omega\tau}\overline{f(\tau)}d\tau\bigg)
w_{\Omega}(\omega)d\omega\\
=\int\limits_{\mathbb{R}}\hat{f}(\omega)e^{-\frac{\pi}{\alpha}\omega^2}
\bigg(\int\limits_{\mathbb{R}}\sqrt{\alpha}e^{-\pi\alpha(\omega-\lambda)^2}
\overline{\hat{f}(\lambda)}\,d\lambda\bigg)w_{\Omega}(\omega)\,d\omega.
\end{multline*}
Comparing the last equality with \eqref{ChOrd2}, we see that
\begin{multline*}
\int\limits_{\mathbb{R}} \bigg(\int\limits_{\mathbb{R}}\hat{f}(\omega)e^{-\frac{\pi}{\alpha}\omega^2}
e^{2\pi{}i\tau\omega}w_{\Omega}(\omega)d\omega\bigg)
e^{-\frac{\pi}{\alpha}\tau^2}\overline{f(\tau)}d\tau\\
=\int\limits_{\mathbb{R}}\hat{f}(\omega)e^{-\frac{\pi}{\alpha}\omega^2}
\bigg(\int\limits_{\mathbb{R}}\sqrt{\alpha}e^{-\pi\alpha(\omega-\lambda)^2}
\overline{\hat{f}(\lambda)}\,d\lambda\bigg)w_{\Omega}(\omega)\,d\omega.
\end{multline*}
Changing the variable \(\lambda\to\omega-\lambda\) in the inner integral of the right hand side
of the previous equality, we come to the equality
\begin{multline*}
\int\limits_{\mathbb{R}} \bigg(\int\limits_{\mathbb{R}}\hat{f}(\omega)e^{-\frac{\pi}{\alpha}\omega^2}
e^{2\pi{}i\tau\omega}w_{\Omega}(\omega)d\omega\bigg)
e^{-\frac{\pi}{\alpha}\tau^2}\overline{f(\tau)}d\tau\\
=\int\limits_{\mathbb{R}}\hat{f}(\omega)e^{-\frac{\pi}{\alpha}\omega^2}
\bigg(\int\limits_{\mathbb{R}}\sqrt{\alpha}e^{-\pi\alpha\lambda^2}
\overline{\hat{f}(\omega-\lambda)}\,d\lambda\bigg)w_{\Omega}(\omega)\,d\omega.
\end{multline*}

In other words,
\begin{equation}
\label{ToApId2}
\int\limits_{\mathbb{R}} \bigg(\int\limits_{\mathbb{R}}\hat{f}(\omega)e^{-\frac{\pi}{\alpha}\omega^2}
e^{2\pi{}i\tau\omega}w_{\Omega}(\omega)d\omega\bigg)
e^{-\frac{\pi}{\alpha}\tau^2}\overline{f(\tau)}d\tau=
\langle{}\hat{f},\mathscr{M}_{\alpha}\mathscr{I}_{\alpha}
\hat{f}\rangle_{L^2(\mathbb{R},w_{\Omega}(\omega)\,d\omega)},
\end{equation}
where \(\mathscr{I}_{\alpha}\) and \(\mathscr{M}_{\alpha}\) are the operators which were
introduced in Definition \ref{DApU}.
Since \(f\in\mathcal{H}_{w_T,w_{\omega}}\), the inclusion \eqref{ATRb} holds.
According to Lemma \ref{ApIdWN}, the family
\(\lbrace\mathscr{M}_{\alpha}\mathscr{I}_{\alpha}\rbrace_{\alpha\to\infty}\) is an approximative identity
in \(L^2(\mathbb{R},w_{\Omega}(\omega)d\omega)\):
\[
\lim_{\alpha\to\infty}\big\|(\mathscr{M}_{\alpha}\mathscr{I}_{\alpha}y)(\omega)
-y(\omega)
\big\|_{L^2(\mathbb{R},w_{\Omega}(\omega)d\omega)}=0, \ \
\forall\,y\in{}L^2(\mathbb{R},w_{\Omega}(\omega)d\omega).
\]
Taking \(\hat{f}(\omega)\) as \(y(\omega)\) in the last equality and comparing
with the equality \eqref{ToApId2}, we conclude that
\begin{equation}
\label{LS2}
\lim_{\alpha\to\infty}
\int\limits_{\mathbb{R}} \bigg(\int\limits_{\mathbb{R}}\hat{f}(\omega)e^{-\frac{\pi}{\alpha}\omega^2}
e^{2\pi{}i\tau\omega}w_{\Omega}(\omega)d\omega\bigg)
e^{-\frac{\pi}{\alpha}\tau^2}\overline{f(\tau)}d\tau=
\langle{}\hat{f},\hat{f}\rangle_{L^2(\mathbb{R},w_{\Omega}(\omega)d\omega)}.
\end{equation}

\textsf{Step 4}. Taking into account the limiting relations \eqref{LS1} and \eqref{LS2}, we pass to
the limit in the equality \eqref{MulIn} as \(\alpha\to\infty\). We obtain that
\begin{equation}
\label{Fin}
\langle{}f,f\rangle_{L^2(\mathbb{R},w_{T}(t)dt)}+
\langle{}\hat{f},\hat{f}\rangle_{L^2(\mathbb{R},w_{\Omega}(\omega)d\omega)}=0\,,
\ \ \textup{ i.e } \ \ \langle{}f,f\rangle_{\mathcal{H}_{w_{T},w_{\Omega}}}=0\,.
\end{equation}
\end{proof}
\section{A generalization.}
 The system of Gaussians is of the form
\begin{equation}
\label{nong}
g_{\alpha,\tau}(t)=\alpha\,g(\alpha(t-\tau)),
\end{equation}
where
\begin{equation}
\label{gau}
g(t)=e^{-\pi{}t^2}.
\end{equation}
We established that under certain non-degeneracy and regularity conditions imposed on
the weight functions \(w_T\) and \(w_\Omega\) the  system of Gaussians
\(g_{\alpha,\tau}\), where \(\alpha\) runs over \(\mathbb{R}_{+}\) and \(\tau\) runs over
\(\mathbb{R}\), is a complete system in the space \(\mathcal{H}_{w_T,w_{\Omega}}\).
However our reasoning remains true for more general functions \(g(t)\) than the function~%
\eqref{gau}. Let us formulate the appropriate generalization.

As before, we assume that the weight functions \(w_T\) and \(w_{\Omega}\) satisfy the non-degeneracy
condition (Definition \ref{DeCo}). We assume  also that the functions \(1+w_{T}(t)\),
\(1~+~w_{\Omega}(\omega)\) satisfy the regularity conditions (Definition \ref{DefReCo}).
In our formulation, the modules of continuity \(M_{1+w_{T}}\), \(M_{1+w_{\Omega}}\)
 (Definition \ref{DMMC}) corresponding to to the weight functions \(1+w_{T}\), \(1+w_{\Omega}\)
appear.
\begin{theorem} {\ }\\
We assume that a function \(g:\,\mathbb{R}\to\mathbb{C}\), \(g\in{}L^{1}(\mathbb{R},dt)\), is given which satisfy the conditions
\begin{subequations}
\label{ApIde}
\begin{align}
\label{ApIde1}
\int\limits_{\mathbb{R}}g(t)\,dt=1;\\
\label{ApIde2}
\lim_{\alpha\to\infty}\int\limits_{\mathbb{R}\setminus(-\delta,\delta)}
\alpha\,|g(\alpha\eta)|\sqrt{M_{1+w_T}(|\eta|)}\,d\eta=0& \ \ \ \textup{for each} \ \ \delta>0;\\
\label{ApIde3}
\lim_{\alpha\to\infty}\int\limits_{\mathbb{R}\setminus(-\delta,\delta)}
\alpha\,|g(\alpha\eta)|\sqrt{M_{1+w_{\Omega}}(|\eta|)}\,d\eta=0& \ \ \ \textup{for each} \ \ \delta>0.
\end{align}
\end{subequations}
Moreover we assume that the following conditions are fulfilled.
\begin{subequations}
\label{KT}
\begin{gather}
\label{KT1}
\int\limits_{\mathbb{R}}|g(\alpha\eta)|\,M_{1+w_{T}}(|\eta|)\,d\eta<\infty \ \ \
\textup{for each} \ \ \alpha>0;\\
\label{KT2}
\sup_{\eta\in\mathbb{R}}|\hat{g}(\alpha\eta)|\,M_{1+w_{T}}(|\eta|)<\infty \ \ \
\textup{for each} \ \ \alpha>0;
\end{gather}
\end{subequations}
and
\begin{subequations}
\label{KO}
\begin{gather}
\label{KO1}
\int\limits_{\mathbb{R}}|\hat{g}(\alpha\eta)|\,M_{1+w_{\Omega}}(|\eta|)\,d\eta<\infty \ \ \
\textup{for each} \ \ \alpha>0;\\
\label{KO2}
\sup_{\eta\in\mathbb{R}}|\hat{g}(\alpha\eta)|\,M_{1+w_{\Omega}}(|\eta|)<\infty \ \ \
\textup{for each} \ \ \alpha>0;
\end{gather}
\end{subequations}
Then
\begin{enumerate}
\item[\textup{1}.]
For each \(\alpha\in\mathbb{R}_{+}\), \(\tau\in\mathbb{R}\), the function \(g(\alpha(t-\tau))\)
belongs to the space \(\mathcal{H}_{w_T,w_{\Omega}}\).
\item[\textup{2}.]
The system of functions \(\lbrace{}g(\alpha(t-\tau))\rbrace_{\alpha,\tau}\), where \(\alpha\) runs over
\(\mathbb{R}_{+}\), \(\tau\) runs over \(\mathbb{R}\), is a complete system in the space
\(\mathcal{H}_{w_T,w_{\Omega}}\).
\end{enumerate}
\end{theorem}
\vspace{2.0ex}
\noindent
{\large\textbf{Comments}.}\\
1. The conditions \eqref{ApIde} ensure that the operator family \(\lbrace\mathscr{I}_{\alpha}\rbrace_{\alpha\in\mathbb{R}_{+}}\) (Definition \ref{DApU})
is an approximative identity
in each of the spaces \(L^2(\mathbb{R},w_{T}(t)dt\) and \(L^2(\mathbb{R},w_{\Omega}(\omega)d\omega\).\\
2. Since \(M_w(\eta)\geq1\) for any weight function \(w\), the condition \eqref{KO1} implies
that \(\int\limits_{\mathbb{R}}|\hat{g}(\eta)|\,d\eta<\infty\). \ \
Hence \(\sup_{\eta\in\mathbb{R}}|g(\eta)|<\infty\).
 From \eqref{KT1} it follows now that \(g(\alpha(t-\tau))\in{}L^2(\mathbb{R},w_{T}(t)dt)\)
 for each \(\alpha\in\mathbb{R}_{+},\tau\in\mathbb{R}\). From \eqref{KO} it follows that
\(\hat{g}(\alpha^{-1}\omega)\in{}L^2(\mathbb{R},w_{\Omega}(\omega)d\omega)\)
for each \(\alpha\in\mathbb{R}_{+}\). Thus  the function \(g(\alpha(t-\tau))\)
belongs to the space \(\mathcal{H}_{w_T,w_{\Omega}}\) for each \(\alpha\in\mathbb{R}_{+}\), \(\tau\in\mathbb{R}\).\\
3. The conditions \eqref{KT} are used to prove the convergence of the double integral analogous to the integral in \eqref{Ker}.\\
4. The congitions \eqref{KO} are used to prove the convergence of the double integral analogous to the integral in \eqref{Ker2}.

\end{document}